\documentclass[11pt]{amsart}
\usepackage{color}
\newtheorem{theorem}{Theorem}[section]
\newtheorem{lemma}[theorem]{Lemma}
\newtheorem{corollary}[theorem]{Corollary}

\theoremstyle{definition}

\theoremstyle{remark}

\numberwithin{equation}{section}

\begin{document}

\title[Inverse Gauss curvature flow and $p$-capacitary Orlicz-Minkowski problem]{ An inverse Gauss curvature flow and its application to $p$-capacitary Orlicz-Minkowski problem }


\author{Bin Chen}
\address{Bin Chen, Xia Zhao, and Peibiao Zhao \newline \indent School of Mathematics and Statistics, Nanjing University of Science and Technology, Nanjing, China}
\curraddr{}
\email{chenb121223@163.com}
\email{zhaoxia20161227@163.com}
\email{pbzhao@njust.edu.cn}
\thanks{}

\author{Weidong Wang}
\address{Weidong Wang \newline \indent Three Gorges Mathematical Research Center, China Three Gorges University, Yichang, China}
\curraddr{}
\email{wangwd722@163.com}
\thanks{Supported  by the Natural Science
Foundation of China (No.12271254; 12141104)}

\author{Xia Zhao}
\address{}
\curraddr{}
\email{}
\thanks{}

\author{Peibiao Zhao}
\address{}
\curraddr{}
\email{}
\thanks{Corresponding author: Peibiao Zhao}



\date{}

\dedicatory{}

\begin{abstract}
  In [Calc. Var., 57:5 (2018)], Hong-Ye-Zhang proposed the $p$-capacitary Orlicz-Minkowski problem and proved the existence of convex solutions to this problem by variational method for $p\in(1,n)$.
  However, the smoothness and uniqueness of solutions are still open.

  Notice that the $p$-capacitary Orlicz-Minkowski problem can be converted equivalently to a Monge-Amp\`{e}re type equation in smooth case:
  \begin{align}\label{0.1}
  f\phi(h_K)|\nabla\Psi|^p=\tau G
  \end{align}
  for $p\in(1,n)$ and some constant $\tau>0$, where $f$ is a positive function defined on the unit sphere $\mathcal{S}^{n-1}$, $\phi$ is a continuous positive function defined in $(0,+\infty)$, and $G$ is the Gauss curvature.

  In this paper, we confirm the existence of smooth solutions to $p$-capacitary Orlicz-Minkowski problem  with $p\in(1,n)$ for the first time by  a class  of inverse Gauss curvature flows, which  converges smoothly to the solution of Equation (\ref{0.1}).
  Furthermore, we prove the uniqueness result for Equation (\ref{0.1}) in a special case.
\end{abstract}

\maketitle

\section{ Introduction}
The classical Brunn-Minkowski-theory (abbreviated as BMT) of convex bodies (i.e., a compact, convex sets) in $n$-dimensional Euclidean spaces $\mathbb{R}^n$ plays an important role in the study of convex geometric analysis and develops rapidly in recent years.
The classical Minkowski problem is one of the main cornerstones of the classical BMT (one can see \cite{G,S} for details).
This problem  is to find a convex body $K$ in $\mathbb{R}^n$ with the prescribed surface area measure $S(K,\cdot)$, which is induced by the volume variation, that is, for each convex body $L$, there holds
\begin{align}\label{1.0}
\frac{d}{dt}V(K+tL)\bigg|_{t=0}=\int_{\mathcal{S}^{n-1}}
h(L,\cdot)dS(K,\cdot),
\end{align}
where $K+tL=\{x+ty: x\in K, y\in L\}$ is the Minkowski sum, $\mathcal{S}^{n-1}$ is the unit sphere, and $h(L,\cdot)$ is the support function of convex body $L$ in $\mathbb{R}^n$.

With the development of the classical BMT, it has inspired many other theory of a similar nature. Examples include the $L_{\mathfrak{p}}$ BMT, Orlicz BMT and their dual theory. Of course, the related Minkowski-type problems are naturally important research contents, see e.g., \cite{CH,GHW,HLY,HH,Hu,Li,L,LY,LYZ,XH} and the references
therein.

The Minkowski type problem for the measure associated with the solution to the boundary-value problem is doubtless an extremely important variant, as some typical examples,
we refer to seminal papers \cite{CF,J} on {\it capacity} and {\it torsional rigidity} by Jerison and Colesanti-Fimiani and subsequent progress, e.g., \cite{CD,CNS,FZH,HYZ,WH,XX,ZX1}.

In the present paper, we will further study the $p$-capacitary Minkowski problem for Orlicz case proposed by Hong-Ye-Zhang \cite{HYZ}.
To describe this type of problem, we first recall the definition of the $p$-capacity functional and its variational formula.

For $p\in(1,n)$, the electrostatic $p$-capacity of a convex body $K$ in $\mathbb{R}^n$ is described by (see \cite{CNS})
$$C_p(K)=\inf\bigg\{\int_{\mathbb{R}^n}|\nabla\psi|^pdx: \psi\in C_c^\infty(\mathbb{R}^n),\ \ \psi\geq1\ on\ K\bigg\},$$
where $C_c^{\infty}(\mathbb{R}^n)$ denotes the set of all infinitely differentiable functions with compact support in $\mathbb{R}^n$, and $\nabla\psi$ denotes the gradient of $\psi$. The geometric quantity $C_2(K)$ is the classical electrostatic (or Newtonian) capacity of $K$ (see \cite{J}).

Let $K$ be a convex body and $p\in(1,n)$. The $p$-equilibrium potential $\Psi$ of $K$ is the unique solution to the following boundary value problem (see \cite{LN})
\begin{align}\label{1.1}
\begin{cases}
\triangle_p\Psi=0\ \ \ \ in\ \mathbb{R}^n\backslash K,\\
\Psi=1\ \ \ \ \ on\ \partial K,\\
\Psi(x)\rightarrow0\ \ \ \ as\ |x|\rightarrow\infty,
\end{cases}
\end{align}
where
$$\triangle_p\Psi=div(|\nabla\Psi|^{p-1}\nabla\Psi)$$
is the $p$-Laplace operator.

Similar to the volume variational formula (\ref{1.0}),
Colesanti et al (\cite{CNS}) established the variational formula for $p$-capacity as follows:
For two convex bodies $K, L$ and $p\in(1,n)$, then
\begin{align}\label{1.2.0}
\frac{d}{dt}C_p(K+tL)\bigg|_{t=0}=(p-1)
\int_{\mathcal{S}^{n-1}}
h(L,\xi)d\mu_p(K,\xi),
\end{align}
and the Poincar\'{e} $p$-capacity formula
\begin{align}\label{1.2.1}
C_p(K)=\frac{p-1}{n-p}\int_{\mathcal{S}^{n-1}}
h(K,\xi)d\mu_p(K,\xi),
\end{align}
where $\mu_p(K,\cdot)$ is a finite Borel measure on $\mathcal{S}^{n-1}$, called the electrostatic $p$-capacitary measure of $K$, defined by
\begin{align}\label{1.2}
\mu_p(K,\eta)=\int_{g_K^{-1}(\eta)}
|\nabla\Psi|^pd\mathcal{H}^{n-1}
=\int_{\eta}|\nabla\Psi|^pdS(K,\cdot),
\end{align}
for each Borel set $\eta\subset\mathcal{S}^{n-1}$,
where $g_K$ is the Gauss map, and $\mathcal{H}^{n-1}$ is the $(n-1)$-dimensional Hausdorff measure.

Naturally, the {\it $p$-capacitary Minkowski problem} can be proposed as: Let $\mu$ be a finite Borel measure on $\mathcal{S}^{n-1}$ and $p\in(1,n)$.
Under what necessary and sufficient conditions does there exist a (unique) convex body $K$ in $\mathbb{R}^n$ such that
$$d\mu_p(K,\cdot)=d\mu?$$
When $p=2$, this problem has been solved in the seminal paper by Jerison \cite{J}.
A convex solution of this problem for $p\in(1,2)$ was obtained in \cite{CNS}.
The case of all $p\in(1,n)$ was recently solved by Akman et al in their groundbreaking work \cite{AG}.

As an extension of the $p$-capacitary Minkowski problem, the Orlicz case was introduced by Hong et al \cite{HYZ}.
It is stated as: under what conditions on a given function $\phi$ and a given Borel measure $\mu$ defined on $\mathcal{S}^{n-1}$, one can find a convex body $K$ in $\mathbb{R}^n$ such that
\begin{align}\label{1.3}
\tau\phi(h(K,\cdot))d\mu_p(K,\cdot)=d\mu,
\end{align}
for some constant $\tau>0$?

For this problem, Hong et al solved the existence of solutions with $p\in(1,n)$ for both discrete and general measures under some mild conditions.
When $\phi(h)=h^{1-\mathfrak{p}}$ in (\ref{1.3}) for $\mathfrak{p}\in\mathbb{R}$, it is the $L_{\mathfrak{p}}$ $p$-capacitary Minkowski problem introduced in \cite{ZX1}. There are many results for different ranges of $\mathfrak{p}$ and $p$, for instance, when $p\in(1,n)$ and $\mathfrak{p}\in(1,\infty)$, the even convex solution was obtained in \cite{ZX1}.
When $p\in(1,2)$, $\mathfrak{p}\in(0,1)$ and $p\geq n$, $\mathfrak{p}\in(0,1)$, the polytopal solutions are proved in \cite{XX} and \cite{LX}, respectively. Feng et al \cite{FZH} supplemented the case of $\mathfrak{p}\in(0,1)$ and $p\in(1,n)$ for general measures. When $\mathfrak{p}=0$ and $p\in(1,n)$ in (\ref{1.3}), it is the logarithmic Minkowski problem for $p$-capacity, and its polytopal solution is proved in \cite{WH}.

It is worth noting that the smoothness of solutions to the Minkowski-type problems has always been an important  issue.
To the best of our knowledge, the smooth solutions of the $p$-capacitary Minkowski-type problems have not been resolved.

Motivated by the celebrated works and the related topics stated above, we thus in the present paper try to investigate and confirm the existence of the non-symmetric smooth solutions to the $p$-capacitary Orlicz-Minkowski problem. One of the main methods used in this paper to solve this problem is the so-called  inverse Gauss curvature flow method.

A basic idea of using inverse Gauss curvature flow to solve the $p$-capacitary Orlicz-Minkowski problem (\ref{1.3})
 can be summarized as follows.

1)  The  $p$-capacitary Orlicz-Minkowski problem (\ref{1.3}) can be converted to a Monge-Amp\`{e}re type equation equivalently in smooth case below (\cite{HYZ})
\begin{align}\label{1.4}
f\phi(h_K)|\nabla\Psi(g^{-1}_K)|^p=\tau G,
\end{align}
for $p\in(1,n)$ and some constant $\tau>0$, the smooth data function $f: \mathcal{S}^{n-1}\rightarrow(0,\infty)$, and $G$ is the Gauss curvature (see Section 2 for details). In this case, the key of the present paper is to find a convex body $K$ in $\mathbb{R}^n$ with the support function $h_K$ satisfying (\ref{1.4}).

2) The solution to the Monge-Amp\`{e}re type equation (\ref{1.4}) follows the limit of solutions of inverse Gauss curvature flows (\ref{1.5}) constructed below.

The Gauss curvature flow was first introduced and studied by Firey \cite{F} to model the shape change of worn stones.  Since then, many scholars have found that using curvature flow to study the hypersurfaces is a very effective tool, such as solving the Minkowski-type problems and geometric inequalities in convex geometric analysis etc., see e.g., \cite{AC,BIS,CH,CWX,CW,HL,HLW,LWX,LL} and the references therein.

Let $\Omega_0$ be a smooth, closed, and strictly convex hypersurface in $\mathbb{R}^n$ enclosing the origin $o$ in its interior,
that is, there is a sufficient small positive constant $\delta_o$ such that the $\delta_0$-neighbourhood of $o$ being with  $U(o,\delta_o)\subset \Omega_0$.
We consider an inverse Gauss curvature flow of a family of convex hypersurfaces $\{\Omega_t\}$ given by $\Omega_t=F(\mathcal{S}^{n-1},t)$, where $F: \mathcal{S}^{n-1}\times[0,T)\rightarrow\mathbb{R}^n$ is a smooth map satisfying
\begin{align}\label{1.5}
\begin{cases}
\frac{\partial F(\xi,t)}{\partial t}=f(\nu)(F\cdot\nu)\phi(F\cdot \nu)|\nabla\Psi(F,t)|^p\sigma_{n-1}\nu-\gamma(t)F(\xi,t),\\
F(\xi,0)=F_0(\xi),
\end{cases}
\end{align}
where $f$ is a given positive smooth function on $\mathcal{S}^{n-1}$,  $``\cdot"$ is the standard inner product in $\mathbb{R}^n$, $\sigma_{n-1}(\xi,t)$ is the product of the principal curvature radii with $\sigma_{n-1}=\det(\nabla_{ij}h+h\delta_{ij})$, $\nu$ is the out normal of $\Omega_t$ at $F(\xi,t)$,
$T$ is the maximal time for which the solution of (\ref{1.5}) exists, and the scalar function $\gamma(t)$ is given by
\begin{align*}
\gamma(t)=\frac{n-p}{p-1}\frac{C_p(\Omega_t)}
{\int_{\mathcal{S}^{n-1}}h/(f\phi(h))d\xi},
\end{align*}
for $p\in(1,n)$, which is used in Section 3 to investigate two important geometric quantities.

Compared with the geometric flows in \cite{BIS,CH,CWX,CW,LL}, the flow we construct in the present paper is more complex due to its inclusion of functions $\phi$, $|\nabla\Psi|$ and $\gamma(t)$, which is reflected in the fact that priori estimates are more difficult to obtain.

We now present the main results of this paper.

\begin{theorem}\label{t1.2}
Let $f$ be a positive smooth function on $\mathcal{S}^{n-1}$, and $\Omega_0$ be a smooth, closed and strictly convex hypersurface in $\mathbb{R}^n$ enclosing the origin in its interior.
Suppose

1) $p\in(1,n)$;

2) function $\phi: (0,\infty)\rightarrow(0,\infty)$ is smooth;

3) $\varphi(s)=\int_0^s1/\phi(t)dt$ exists for all $s>0$ and $\lim_{s\rightarrow\infty}\varphi(s)=\infty$.

\noindent Then, the flow (\ref{1.5}) has a smooth solution $\Omega_t$ for all time $t>0$.
When $t\rightarrow\infty$, there is a subsequence of $\Omega_t$ that converges in $C^\infty$ to a smooth, closed and strictly convex hypersurface $\Omega_\infty$ whose support function satisfies (\ref{1.4}).
\end{theorem}

As an application, we have

\begin{corollary}\label{c1.3}
Under the assumptions of Theorem \ref{t1.2}, there exists a non-symmetric smooth solution to the $p$-capacitary Orlicz-Minkowski problem (\ref{1.3}) for $p\in(1,n)$.
\end{corollary}

For the general $\phi$, the uniqueness of the solution to the $p$-capacitary Orlicz-Minkowski problem is still open. We consider here a special uniqueness result of the equation (\ref{1.4}) in the case of $\tau=1$.

\begin{theorem}\label{t1.4}
Let $p\in(1,n-1]$ and $\delta\geq1$. If whenever
\begin{align}\label{1.9}
\phi(\delta s)\leq\delta^{p+1-n}\phi(s)
\end{align}
holds for positive $s$, then the solution to the equation (\ref{1.4}) is unique.
\end{theorem}

Moreover, based on the parabolic approximation method, we also provide naturally a proof for the weak solution to the $p$-capacitary Orlicz-Minkowski problem when $p\in(1,n)$, which has been obtained by Hong-Ye-Zhang \cite{HYZ}.

\begin{theorem}\label{t1.3}
Let $\mu$ be a finite Borel measure on $\mathcal{S}^{n-1}$ whose support is not contained in any closed hemisphere and $p\in(1,n)$.
Suppose

1) $\phi: (0,\infty)\rightarrow(0,\infty)$ is a continuous function;

2) $\varphi(s)=\int_0^s1/\phi(t)dt$ exists for all $s>0$ and $\lim_{s\rightarrow\infty}\varphi(s)=\infty$.

\noindent Then, there exists a convex body $\Omega$ such that the equation (\ref{1.3}) holds.
\end{theorem}

The organization of this paper is as follows. In Section 2, the corresponding background materials are introduced. In Section 3, we introduced geometric flow and its associated functional. In Section 4, we establish the priori estimates for the solution to the flow (\ref{1.5}). Finally, we prove the main results in Section 5.

\section{ Preliminaries}
In this section, we list some necessary  facts about convex hypersurfaces that readers can refer to \cite{U} and a celebrated book of Schneider \cite{S} for details.  Let $\mathbb{R}^n$ be the $n$-dimensional Euclidean
space, $\mathcal{S}^{n-1}$ be the unit sphere in $\mathbb{R}^n$, and $\Omega$ be a smooth, closed and strictly convex hypersurface containing the origin in its interior.
The support function of $\Omega$ is defined by
$$h_{\Omega}(\xi)=h(\Omega,\xi)=\max\{\xi\cdot Y: Y\in\Omega\},\ \ \forall\xi\in\mathcal{S}^{n-1}.$$
For $\pm v\in\mathcal{S}^{n-1}$, the support function of line segment $\overline{v}$ joining the points $\pm v$ is defined as
$$h(\overline{v},\xi)=|\xi\cdot v|,\ \ \xi\in\mathcal{S}^{n-1}.$$
The radial function of $\Omega$ is defined by
$$r_{\Omega}(\upsilon)=r(\Omega,\upsilon)=\max\{c>0: c\upsilon\in\Omega\}, \ \ \upsilon\in\mathcal{S}^{n-1}.$$
Obviously, $r_\Omega(\upsilon)\upsilon\in\partial\Omega$.

Let $g: \partial\Omega\rightarrow\mathcal{S}^{n-1}$ be the Gauss map of $\Omega$. For $\xi\in\mathcal{S}^{n-1}$, the inverse Gauss map, denoted by $g^{-1}$, is expressed as
$$g^{-1}(\xi)=F(\xi)=\{X\in\partial\Omega:
g(X)\ is\ well\ defined\ and\ g(X)\in\{\xi\}\}.$$
Specially, for a convex hypersurface $\Omega$ of class $C^2_+$ ($\partial\Omega$ is $C^2$ smooth and has positive Gauss curvature), then the support function of $\Omega$ can be stated as
$$h(\Omega,\xi)=\xi\cdot g^{-1}(\xi)=g(X)\cdot X,\ \ X\in\partial\Omega.$$
Furthermore, the gradient of $h(\Omega,\cdot)$ satisfies
\begin{align}\label{2.01}
\nabla h(\Omega,\xi)=g^{-1}(\xi).
\end{align}

Let $e=\{e_{ij}\}$ be the standard metric of $\mathcal{S}^{n-1}$.
The second fundamental form of $\Omega$ is defined as
\begin{align}\label{2.1}
\Pi_{ij}=\nabla_{ij}h+he_{ij},
\end{align}
where $\nabla_{ij}$ is the second order covariant derivative with respect to $e_{ij}$.
By the Weingarten's formula and (\ref{2.1}), the principal radii of $\Omega$, under a smooth local orthonormal frame on $\mathcal{S}^{n-1}$, are the eigenvalues of matrix
\begin{align}\label{2.2}
b_{ij}=\nabla_{ij}h+h\delta_{ij}.
\end{align}
Particularly, the Gauss curvature of $F(\xi)$ can be expressed as
\begin{align}\label{2.3}
G(\xi)=\frac{1}{\det(\nabla_{ij}h+h\delta_{ij})}.
\end{align}

Next, we introduce the Orlicz norm, one can see \cite{HLY} for details. Let $\varphi: [0,\infty)\rightarrow[0,\infty)$ be continuous, strictly increasing, continuously differentiable on $(0,\infty)$ with positive derivative, and satisfy the assumption in Theorem \ref{t1.3}, $\mu$ be a finite Borel measure on $\mathcal{S}^{n-1}$, and $\mathfrak{f}: \mathcal{S}^{n-1}\rightarrow[0,\infty)$ be a continuous function.

The Orlicz norm $\|\mathfrak{f}\|_{\varphi,\mu}$ is defined by
\begin{align}\label{2.4}
\|\mathfrak{f}\|_{\varphi,\mu}=\inf\bigg\{\lambda>0: \frac{1}{|\mu|}
\int_{\mathcal{S}^{n-1}}\varphi
\bigg(\frac{\mathfrak{f}}{\lambda}\bigg)
d\mu\leq\varphi(1)\bigg\},
\end{align}
where $|\mu|=\mu(\mathcal{S}^{n-1})$.

Moreover, the Orlicz norm satisfies the following properties:
$$\|c\mathfrak{f}\|_{\varphi,\mu}
=c\|\mathfrak{f}\|_{\varphi,\mu},\ \ c\geq0,$$
and
\begin{align}\label{2.5}
\mathfrak{f}\leq \mathfrak{g}\ \ \Rightarrow\ \ \|\mathfrak{f}\|_{\varphi,\mu}\leq
\|\mathfrak{g}\|_{\varphi,\mu}.
\end{align}
If $\mu(\{\mathfrak{f}\neq0\})>0$, the Orlicz norm $\|\mathfrak{f}\|_{\varphi,\mu}>0$ and
$$\|\mathfrak{f}\|_{\varphi,\mu}=\lambda_0\ \ \Leftrightarrow\ \
\frac{1}{|\mu|}\int_{\mathcal{S}^{n-1}}\varphi
\bigg(\frac{\mathfrak{f}}{\lambda_0}\bigg)d\mu=\varphi(1).$$

\section{Inverse curvature flow and its associated functional}

For convenience, the curvature flow is restated here. Let $\Omega_0$ be a smooth, closed, and strictly convex hypersurface in $\mathbb{R}^n$ enclosing the origin in its interior.
We consider the following inverse Gauss curvature flow
\begin{align}\label{3.1}
\begin{cases}
\frac{\partial F(\xi,t)}{\partial t}=f(\nu)(F\cdot \nu)\phi(F\cdot\nu)|\nabla\Psi(F,t)|^p
\sigma_{n-1}\nu-\gamma(t)F(\xi,t),\\
F(\xi,0)=F_0(\xi),
\end{cases}
\end{align}
where the scalar function $\gamma(t)$ is given by
\begin{align}\label{3.2}
\gamma(t)=\frac{n-p}{p-1}\frac{C_p(\Omega_t)}
{\int_{\mathcal{S}^{n-1}}h/(f\phi(h))d\xi},
\end{align}
for $p\in(1,n)$.
As discussed in Section 2, the support function of $\Omega_t$ can be expressed as $h(\xi,t)=\xi\cdot F(\xi,t)$, we thus derive the evolution equation for $h(\cdot,t)$ along the flow (\ref{3.1}) as follows
\begin{align}\label{3.3}
\begin{cases}
\frac{\partial h(\xi,t)}{\partial t}=f(\xi)h\phi(h)|\nabla\Psi(F,t)|^p\sigma_{n-1}
-\gamma(t)h(\xi,t),\\
h(\xi,0)=h_0(\xi).
\end{cases}
\end{align}

Now we investigate the characteristics of two more important geometric functionals that are key to proving the long-time existence of solutions to Equation (\ref{3.3}).

\begin{lemma}\label{l3.1}
Suppose that the convex body $K_t$ contains the origin in its interior with $\Omega_t=\partial K_t$, and $p\in(1,n)$.
Then the $p$-capacity $C_p(\Omega_t)$ is monotone non-decreasing along the flow (\ref{3.1}).
\end{lemma}

\begin{proof}
Let $\Psi(F,t)$ is the $p$-equilibrium potential of $K_t$. Theorem 3.5 in \cite{CNS} showed that
\begin{align*}
\partial_tC_p(\Omega_t)&=\frac{p-1}{n-p}\partial_t
\bigg(\int_{\mathcal{S}^{n-1}}
h(\xi,t)|\nabla\Psi(F(\xi,t),t)|^p\sigma_{n-1}d\xi\bigg)\\
&=(p-1)\int_{\mathcal{S}^{n-1}}
|\nabla\Psi(F,t)|^p\sigma_{n-1}\partial_th(\xi,t)d\xi.
\end{align*}
From (\ref{3.3}), by H\"{o}lder inequality, we obtain
\begin{align*}
\partial_tC_p(\Omega_t)
&=(p-1)\int_{\mathcal{S}^{n-1}}
|\nabla\Psi(F,t)|^p\sigma_{n-1}\partial_thd\xi\\
&=(p-1)\bigg(\int_{\mathcal{S}^{n-1}}fh\phi(h)
|\nabla\Psi|^{2p}\sigma_{n-1}^2d\xi
-\gamma(t)\int_{\mathcal{S}^{n-1}}h
|\nabla\Psi|^{p}\sigma_{n-1}d\xi\bigg)\\
&=\frac{p-1}{\int_{\mathcal{S}^{n-1}}\frac{h}{f\phi(h)}d\xi}
\bigg[\int_{\mathcal{S}^{n-1}}fh\phi(h)
|\nabla\Psi|^{2p}\sigma_{n-1}^2d\xi
\int_{\mathcal{S}^{n-1}}\frac{h}{f\phi(h)}d\xi\\
&\ \ \ -\bigg(\int_{\mathcal{S}^{n-1}}h
|\nabla\Psi|^{p}\sigma_{n-1}d\xi\bigg)^2\bigg]\\
&\geq\frac{p-1}{\int_{\mathcal{S}^{n-1}}\frac{h}{f\phi(h)}d\xi}
\bigg[\bigg(\int_{\mathcal{S}^{n-1}}h
|\nabla\Psi|^{p}\sigma_{n-1}d\xi\bigg)^2
-\bigg(\int_{\mathcal{S}^{n-1}}h
|\nabla\Psi|^{p}\sigma_{n-1}d\xi\bigg)^2\bigg]\\
&=0.
\end{align*}
According to the equality condition of H\"{o}lder inequality, we see that the equality of inequality holds if and only if $h(\cdot,t)$ solves the equation $f\phi(h)|\nabla\Psi|^p\sigma_{n-1}=\tau$ for some constant $\tau>0$.
\end{proof}

\begin{lemma}\label{l3.2}
Suppose that the function
$\varphi(\cdot)$ satisfies the assumption in Theorem \ref{t1.2}, and $p\in(1,n)$.
Define the functional $$\Phi(t):=\Phi(\Omega_t)=\int_{\mathcal{S}^{n-1}}
\frac{\varphi(h)}{f(\xi)}d\xi.$$
Then, along the flow (\ref{3.3}), the functional $\Phi(t)$ keeps unchanged, that is, $\Phi(t)\equiv R$ for some positive constant $R$.
\end{lemma}

\begin{proof}
From (\ref{3.2}), (\ref{3.3}) and the definition of $\varphi(\cdot)$, we obtain
\begin{align*}
\partial_t\Phi(t)
&=\int_{\mathcal{S}^{n-1}}\frac{\partial_th}{f\phi(h)}d\xi\\
&=\int_{\mathcal{S}^{n-1}}\bigg(f\phi(h)h|\nabla\Psi|^p\sigma_{n-1}
-\gamma(t)h\bigg)\frac{1}{f\phi(h)}d\xi\\
&=\int_{\mathcal{S}^{n-1}}h|\nabla\Psi|^p\sigma_{n-1}d\xi-
\frac{(n-p)C_p(\Omega_t)}{(p-1)\int_{\mathcal{S}^{n-1}}\frac{h}
{f\phi(h)}d\xi}\int_{\mathcal{S}^{n-1}}\frac{h}
{f\phi(h)}d\xi\\
&=0.
\end{align*}
Thus, the proof of the lemma is completed.
\end{proof}

Next, we give the evolution equation of $\Psi(F,t)$.

\begin{lemma}\label{l3.3}
Suppose that the convex body $K_t$ contains the origin in its interior with $\Omega_t=\partial K_t$, and $\Psi(F,t)$ is the $p$-equilibrium potential of $K_t$. Then
$$\partial_t\Psi(F(\xi,t),t)=
|\nabla\Psi(F(\xi,t),t)|\partial_th(\xi,t).$$
\end{lemma}

\begin{proof}
Let $h(\xi,t)$ be the support function of $\Omega_t$. From Lemma 3.1 in \cite{CNS}, one can see that $\Psi(F,t)$ is differentiable with respect to $t$. Due to $\Psi(F,t)=1$ on $\Omega_t$, take the derivative of both sides with respect to $t$, thus we have
\begin{align*}
\partial_t\Psi+\nabla\Psi\cdot\partial_tF(\xi,t)=0.
\end{align*}
Here $\partial_tF=\nabla_i(\partial_th)e_i+\partial_th\xi.$
Then, it is further calculated as
\begin{align*}
\partial_t\Psi=-\nabla\Psi\cdot
(\nabla_i(\partial_th)e_i+\partial_th\xi).
\end{align*}
Recall a fact from \cite{CNS} that $|\nabla\Psi(F,t)|=-\nabla\Psi(F,t)\cdot\xi$, thus
\begin{align*}
\partial_t\Psi&=|\nabla\Psi|\xi\cdot
(\nabla_i(\partial_th)e_i+\partial_th\xi)\\
&=|\nabla\Psi|\partial_th.
\end{align*}
Thus the proof is completed.
\end{proof}

\section{A priori estimates}
In this section, we establish the priori estimates for the solution to Equation (\ref{3.3}).

\subsection{$C^0, C^1$-Estimates}

\begin{lemma}\label{l4.1}
Let $h(\cdot,t)$, $t\in[0,T)$, be a non-symmetric smooth solution to Equation (\ref{3.3}), and $T$ be the maximal time for which the smooth solution of Equation (\ref{3.3}) exists.
Under the corresponding assumptions of Theorem \ref{t1.2}, then there exits positive constants $l$ and $L$ independent of t such that
\begin{align}\label{4.1}
l\leq h(\cdot,t)\leq L,
\end{align}
and
\begin{align}\label{4.01}
l\leq r(\cdot,t)\leq L.
\end{align}
\end{lemma}

\begin{proof}
From the definitions of support function and radial function, there is
\begin{align}\label{4.02}
r(\upsilon,t)\upsilon=\nabla h(\xi,t)+h(\xi,t)\xi,
\end{align}
so we only need to prove (\ref{4.1}) (or (\ref{4.01})).

We first deal with the right-hand side of (\ref{4.1}).
Let $T$ be the maximal time for which the smooth solution of Equation (\ref{3.3}) exists. For fixed $t_0\in[0,T)$, suppose that the maximum of $h(\cdot,t_0)$ is attained at $(\xi_{t_0},t_0)$ for $\xi_{t_0}\in\mathcal{S}^{n-1}$. Set
$$\max_{\mathcal{S}^{n-1}}h(\xi,t_0)=h(\xi_{t_0},t_0).$$
Let
$$\hbar \ =\sup_{t_0\in[0,T)}h(\xi_{t_0},t_0).$$
By the convexity of $\Omega_t$ and definition of support function, one has (see e.g., Lemma 2.6 in \cite{CL})
\begin{align}\label{4.1.1}
h(\xi,t_0)\geq\hbar \ \xi_{t_0}\cdot\xi,\ \
\forall\ \xi\in\mathcal{S}^{n-1},
\end{align}
where $h$ and $\hbar$ are on the same hypersurface.

Let $S_{\xi_{t_0}}=\{\xi\in\mathcal{S}^{n-1}: \xi_{t_0}\cdot\xi>0\}$ be the hemisphere containing $\xi_{t_0}$. From the definition of $\varphi$, we know that $\varphi^\prime(h)>0$. By Lemma \ref{l3.2} and (\ref{4.1.1}), we have
\begin{align}\label{4.1.2}
\nonumber\Phi(t)=\Phi(0)&\geq\int_{S_{\xi_{t_0}}}
\frac{\varphi(h)}{f}d\xi\\
&\geq\int_{S_{\xi_{t_0}}}\frac{1}{f}
\varphi(\hbar\xi_{t_0}\cdot\xi)d\xi\\
&\nonumber=\int_{S}\frac{1}{f}\varphi(\hbar\hat{\xi})d\xi,
\end{align}
where  $S=\{\xi\in\mathcal{S}^{n-1}: \hat{\xi}>0\}$. Further, let
$S_{\frac{1}{2}}=\{\xi\in\mathcal{S}^{n-1}: \hat{\xi}\geq\frac{1}{2}\}$.
Since $f$ is a positive smooth function on $\mathcal{S}^{n-1}$, it follows that
$\int_{S_{\frac{1}{2}}}\frac{1}{f}d\xi\hat{=}c_0$ for some positive constant $c_0$.
Thus (\ref{4.1.2}) can be directly deduced
$$\Phi(0)\geq c_0\varphi\bigg(\frac{\hbar}{2}\bigg),$$
which means that $\varphi(\frac{\hbar}{2})\leq R/c_0$ is uniformly bounded. Since $\varphi$ is strictly increasing, we can know that $h(\cdot,t)$ has uniformly positive upper bound.

In the following, we use contradiction to prove the left-hand side of (\ref{4.1}).
Let $h(\cdot,t)\rightarrow0$. Since $h(\cdot,t)$ has the uniformly positive upper bound, it follows from the dominated convergence theorem that
$$S_p(\Omega_t)\rightarrow0,$$
where $S_p(\Omega_t)$ is the total mass of $p$-surface area measure of $\Omega_t$ (see \cite{L}).
However, by Lemma 1 in \cite{LXZ}, Lemma \ref{l3.1} and $p\in(1,n)$, there is $$S_p(\Omega_t)\geq\bigg(\frac{p-1}{n-p}\bigg)^{p-1}C_p(\Omega_t)
\geq\bigg(\frac{p-1}{n-p}\bigg)^{p-1}C_p(\Omega_0)>0.$$
This is a  contradiction. Thus $h(\cdot,t)$ has uniformly positive lower bound.
\end{proof}

Combining Lemma \ref{l4.1} with the convexity of $\Omega_t$
yields the $C^1$-estimates as follows.

\begin{lemma}\label{l4.3}
Under the assumptions of Lemma \ref{l4.1}, we obtain
$$|\nabla h(\cdot,t)|\leq L_0,\ \ and \ \
|\nabla r(\cdot,t)|\leq L_0,$$
where $L_0$ is a positive constant depending on Lemma \ref{l4.1}.
\end{lemma}

\begin{proof}
From the formula (\ref{4.02}), one has
$$r^2=h^2+|\nabla h|^2.$$
Thus, by Lemma \ref{l4.1}, the proof of this lemma can be obtained.
\end{proof}

\begin{lemma}\label{l4.4}
Let the convex body $K_t$ contain the origin in its interior with $\Omega_t=\partial K_t$, and let $\Psi(F,t)$ be the $p$-equilibrium potential of $K_t$.
Under the assumptions of Lemma \ref{l4.1}, then there are positive constants $\hat{l}$, $\tilde{l}$ and $\bar{l}$, independent of $t$, such that
\begin{align*}
\hat{l}\leq|\nabla\Psi(\cdot,t)|\leq\tilde{l},\ \
and \ \
|\nabla^k\Psi(\cdot,t)|\leq\bar{l},
\end{align*}
for the positive integer $k\geq2$.
\end{lemma}

\begin{proof}
Since we have already obtain the uniform upper bound of $h(\cdot,t)$, and $\Omega_t$ is smooth,
it follows from Lemma 2.18 in \cite{CNS} that there exists a positive constant $\hat{l}$ depending only on $n,q$ and uniform upper bound of $h(\cdot,t)$, such that
$$|\nabla\Psi|\geq\hat{l}.$$

Next, we will prove that $|\nabla\Psi|\leq\tilde{l}$, which can be found in the proof of Lemma 3.1 in \cite{CNS} or \cite{LS}. For the completeness of the article, we will list the main proof processes here.
Let $x\in\Omega_t$ and note that there exists a ball $B$, included in $\Omega_t$ and internally tangent to $\Omega_t$ at $x$, with
radius $r$ which can be chosen to be independent of $t$ and $x$. Let $\bar{\Psi}$ be the $p$-equilibrium potential of $B$.
Then we have $\Psi(\cdot,t)\geq\bar{\Psi}(\cdot)$ on $\Omega_t$.
By the comparison principle $\Psi(\cdot,t)\geq\bar{\Psi}(\cdot)$ in $\mathbb{R}^n\backslash K_t$, and, since $\Psi(x,t)=\bar{\Psi}(x)$, we have $$|\nabla\Psi(x,t)|\leq|\nabla\bar{\Psi}(x)|.$$
On the other hand the value $|\nabla\bar{\Psi}(x)|$ can be explicitly computed and is a positive constant depending on $r$ and $n$ only. Combined with (\ref{2.01}),
it is easy to conclude that
$$|\nabla\Psi(\nabla h(\xi,t),t)|\leq\tilde{l},\ \ \forall(\xi,t)\in\mathcal{S}^{n-1}\times[0,T).$$

Moreover, by virtue of Schauder's theory (see e.g., Lemmas 6.4 and 6.17 in \cite{GT}), there is a positive constant $\bar{l}$, independent of $t$, satisfying that
$$|\nabla^k\Psi(\nabla h(x,t),t)|\leq\bar{l},\ \ \forall(x,t)\in\mathcal{S}^{n-1}\times[0,T),$$
for the positive integer $k\geq2$.
\end{proof}

As a result of Lemmas \ref{l4.1} and \ref{l4.4}, we can obtain the following corollary.

\begin{corollary}\label{c4.2}
Under the assumptions of Lemma \ref{l4.1}, the scalar function $\gamma(t)$ has uniformly positive upper and lower bounds.
\end{corollary}

\begin{proof}
From lemma \ref{l4.1}, we know that $h(\cdot,t)$ has a uniform positive upper bound $L$, so that the hypersurface $\Omega_t$ generated by $h(\cdot,t)=L$ is enclosed by a sphere with radius $L$. By Lemma \ref{l3.1} and the homogeneity of $p$-capacity, we have
$$C_p(\Omega_t)\leq C_p(B_L)
=\omega_n\bigg(\frac{n-p}{p-1}\bigg)^{p-1}L^{n-p},$$
where $\omega_n$ denotes the surface area of the unit sphere in $\mathbb{R}^n$.
This means that $C_p(\Omega_t)$ has a positive upper bound independent of $t$.

Similarly, Since $h(\cdot,t)$ has a uniform positive lower bound $l$, then the hypersurface $\Omega_t$ generated by $h(\cdot,t)=l$ contains a sphere with radius $l$. By Lemma \ref{l3.1}, we have
$$\omega_n\bigg(\frac{n-p}{p-1}\bigg)^{p-1}l^{n-p}
=C_p(B_l)\leq C_p(\Omega_t).$$
Clear, it implies that
$C_p(\Omega_t)$ has a positive lower bound independent of $t$.

Therefore, by Lemmas \ref{l4.1}, the proof of this lemma can be obtained.
\end{proof}

\subsection{$C^2$-Estimates}

We start by estimating the lower bound of Gauss curvature, which is equivalent to estimating the upper bound of $\sigma_{n-1}(\cdot,t)=\det(\nabla_{ij}h+h\delta_{ij})$.
This estimate can be obtained by considering proper auxiliary function, see, e.g., \cite{I} for similar techniques.
Let $\alpha=fh\phi(h)|\nabla\Psi|^p$.
In order to deal with $\partial_t|\nabla\Psi|$ and simplify the calculation process, the auxiliary function in \cite{I} is obviously no longer effective. Therefore, we need to establish the following auxiliary functions:
$$\Theta(\xi,t)=\frac{1}{1-\lambda\frac{r^2}{2}}
\frac{\alpha\sigma_{n-1}}{h},$$
{for $\lambda>0$ sufficiently small.}

\begin{lemma}\label{l4.5}
Under the assumptions of Lemma \ref{l4.1}, then
$$\sigma_{n-1}(\cdot,t)\leq L_2,$$
where $L_2$ is a positive constant independent of $t$.
\end{lemma}

\begin{proof}
Let $c_{ij}$ be the cofactor matrix of $(h_{ij}+h\delta_{ij})$ with $\sum_{i,j}c_{ij}(h_{ij}+h\delta_{ij})=(n-1)\sigma_{n-1}$.
Suppose the spatial maximum of $\Theta$ is obtained at $(\hat{\xi}_{\hat{t}},\hat{t})$. Then at this point, we have
\begin{align}\label{4.6}
\nabla_i\Theta=0,\ \ i.e., \ \ \nabla_i\bigg(\frac{\alpha\sigma_{n-1}}{h}\bigg)
+\frac{\alpha\sigma_{n-1}}{h}
\frac{\lambda}{1-\lambda\frac{r^2}{2}}
\nabla_i\bigg(\frac{r^2}{2}\bigg)=0,
\end{align}
and
\begin{align}\label{4.7}
\nabla_{ij}\Theta\leq0.
\end{align}

Now we estimate $\Theta$, using (\ref{4.7}), we have
\begin{align}\label{4.8}
\nonumber\partial_t\Theta&\leq\partial_t\Theta-\alpha c_{ij}\nabla_{ij}\Theta\\
&\nonumber=\partial_t\bigg(\frac{1}{1-\lambda\frac{r^2}{2}}
\frac{\alpha\sigma_{n-1}}{h}\bigg)
-\alpha c_{ij}\nabla_{ij}\bigg(\frac{1}{1-\lambda\frac{r^2}{2}}
\frac{\alpha\sigma_{n-1}}{h}\bigg)\\
&\nonumber=\frac{1}{1-\lambda\frac{r^2}{2}}\bigg[\partial_t
\bigg(\frac{\alpha\sigma_{n-1}}{h}\bigg)-\alpha c_{ij}\nabla_{ij}
\bigg(\frac{\alpha\sigma_{n-1}}{h}\bigg)\bigg]\\
&\ \ +\frac{\lambda}{(1-\lambda\frac{r^2}{2})^2}
\frac{\alpha\sigma_{n-1}}{h}\bigg
[\partial_t\bigg(\frac{r^2}{2}\bigg)
-\alpha c_{ij}\nabla_{ij}\bigg(\frac{r^2}{2}\bigg)\bigg]\\
&\nonumber\ \ -2\alpha c_{ij}\frac{\lambda}{(1-\lambda\frac{r^2}{2})^2}
\nabla_i\bigg(\frac{\alpha\sigma_{n-1}}{h}\bigg)\nabla_j
\bigg(\frac{r^2}{2}\bigg)\\
&\nonumber\ \ -2\alpha c_{ij}\frac{\lambda^2}{(1-\lambda\frac{r^2}{2})^3}
\frac{\alpha\sigma_{n-1}}{h}\nabla_i\bigg
(\frac{r^2}{2}\bigg)\nabla_j
\bigg(\frac{r^2}{2}\bigg).
\end{align}
Substituting (\ref{4.6}) into (\ref{4.8}), we have
\begin{align}\label{4.8.1}
\nonumber\partial_t\Theta&\leq
\frac{1}{1-\lambda\frac{r^2}{2}}\bigg[\partial_t
\bigg(\frac{\alpha\sigma_{n-1}}{h}\bigg)-\alpha c_{ij}\nabla_{ij}
\bigg(\frac{\alpha\sigma_{n-1}}{h}\bigg)\bigg]\\
&\ \ +\frac{\lambda}{(1-\lambda\frac{r^2}{2})^2}
\frac{\alpha\sigma_{n-1}}{h}\bigg
[\partial_t\bigg(\frac{r^2}{2}\bigg)
-\alpha c_{ij}\nabla_{ij}\bigg(\frac{r^2}{2}\bigg)\bigg].
\end{align}

Next, we need to calculate $\partial_t
(\frac{\alpha\sigma_{n-1}}{h})-\alpha c_{ij}\nabla_{ij}
(\frac{\alpha\sigma_{n-1}}{h})$ and $\partial_t(\frac{r^2}{2})
-\alpha c_{ij}\nabla_{ij}(\frac{r^2}{2})$.
Before this, we calculate some facts as follows.
\begin{align*}
&\sigma_{n-1}\partial_t\alpha+\alpha\partial_t\sigma_{n-1}\\
&=\sigma_{n-1}(f\phi(h)|\nabla\Psi|^p\partial_th
+fh|\nabla\Psi|^p\phi^\prime(h)\partial_th
+fh\phi(h)\partial_t|\nabla\Psi|^p)+\alpha c_{ij}(\nabla_{ij}(\partial_th)+\delta_{ij}\partial_th)\\
&=f\phi(h)|\nabla\Psi|^p\sigma_{n-1}(\alpha\sigma_{n-1}-\gamma h)+fh|\nabla\Psi|^p\phi^\prime(h)\sigma_{n-1}
(\alpha\sigma_{n-1}-\gamma h)\\
&\ \ +pfh\phi(h)\sigma_{n-1}|\nabla\Psi|^{p-1}
\partial_t|\nabla\Psi|
+\alpha c_{ij}\nabla_{ij}(\alpha\sigma_{n-1})
+\alpha^2c_{ij}\delta_{ij}\sigma_{n-1}
-\gamma\alpha c_{ij}(\nabla_{ij}h+h\delta_{ij}).
\end{align*}
Here
\begin{align}\label{4.9}
\nonumber\partial_t|\nabla\Psi|&=-[(\nabla^2\Psi)\xi
\cdot(\nabla_i(\partial_th)e_i
+\partial_th\xi)]-\nabla(\partial_t\Psi)\cdot\xi\\
&=-(\nabla^2\Psi)\xi\cdot[\nabla_i(\alpha\sigma_{n-1}-\gamma h)e_i+(\alpha\sigma_{n-1}-\gamma h)\xi]-\nabla(\partial_t\Psi)\cdot\xi.
\end{align}
From Lemma \ref{l3.3}
\begin{align}\label{4.9.1}
\nonumber\nabla(\partial_t\Psi)\cdot\xi&
=\nabla(|\nabla\Psi|\partial_th)\cdot\xi\\
&\nonumber=(|\nabla\Psi|^{-1}\nabla\Psi\nabla^2\Psi\cdot\xi)
(\alpha\sigma_{n-1}-\gamma h)+|\nabla\Psi|\partial_t(\nabla h)\cdot\xi\\
&=(|\nabla\Psi|^{-1}\nabla\Psi\nabla^2\Psi\cdot\xi)
(\alpha\sigma_{n-1}-\gamma h)+|\nabla\Psi|\partial_t(\nabla_ihe_i+h\xi)\cdot\xi\\
&\nonumber=(|\nabla\Psi|^{-1}\nabla\Psi\nabla^2\Psi\cdot\xi)
(\alpha\sigma_{n-1}-\gamma h)+|\nabla\Psi|(\alpha\sigma_{n-1}-\gamma h).
\end{align}
Substituting (\ref{4.9.1}) into (\ref{4.9}), we have
\begin{align*}
\partial_t|\nabla\Psi|
&=-(\nabla^2\Psi)\xi\cdot\nabla_i(\alpha\sigma_{n-1}-\gamma h)e_i-
(\alpha\sigma_{n-1}-\gamma h)(\nabla^2\Psi)\xi\cdot\xi\\
&\ \ \ \ -(|\nabla\Psi|^{-1}\nabla\Psi\nabla^2\Psi\cdot\xi)
(\alpha\sigma_{n-1}-\gamma h)-|\nabla\Psi|(\alpha\sigma_{n-1}-\gamma h).
\end{align*}
This, together with the fact $\sum_{i,j}c_{ij}(\nabla_{ij}h+h\delta_{ij})=(n-1)\sigma_{n-1}$, yields
\begin{align}\label{4.9.2}
&\nonumber\sigma_{n-1}\partial_t\alpha
+\alpha\partial_t\sigma_{n-1}\\
&\nonumber=\bigg(\frac{1}{h}+\frac{\phi^\prime(h)}
{\phi(h)}\bigg)(\alpha\sigma_{n-1})^2
-\bigg(n+h\frac{\phi^\prime(h)}{\phi(h)}\bigg)
\gamma\alpha\sigma_{n-1}+\alpha^2\sigma_{n-1} c_{ij}\delta_{ij}
+\alpha c_{ij}\nabla_{ij}(\alpha\sigma_{n-1})\\
&\ \ -p\alpha\sigma_{n-1}|\nabla\Psi|^{-1}
[(\nabla^2\Psi)\xi\cdot
(\beta\alpha\sigma_{n-1}r\nabla_ir)e_i
+(\alpha\sigma_{n-1}-\gamma h)(\nabla^2\Psi)\xi\cdot\xi]\\
&\nonumber\ \ -p\alpha\sigma_{n-1}[(|\nabla\Psi|^{-2}
\nabla\Psi\nabla^2\Psi\cdot\xi)
(\alpha\sigma_{n-1}-\gamma h)+
(\alpha\sigma_{n-1}-\gamma h)],
\end{align}
and
\begin{align}\label{4.9.3}
\nonumber\nabla_{ij}\bigg(\frac{\alpha\sigma_{n-1}}{h}\bigg)
&=\frac{1}{h}\nabla_{ij}(\alpha\sigma_{n-1})
-\frac{1}{h^2}(\alpha\sigma_{n-1})\nabla_{ij}h\\
&\ \ -\frac{2}{h^2}\nabla_i(\alpha\sigma_{n-1})\nabla_jh
+\frac{2}{h^3}(\alpha\sigma_{n-1})\nabla_ih\nabla_jh.
\end{align}
From (\ref{4.9.2}) and (\ref{4.9.3}), we have
\begin{align}\label{4.10}
&\nonumber\partial_t\bigg(\frac{\alpha
\sigma_{n-1}}{h}\bigg)-\alpha c_{ij}\nabla_{ij}\bigg(\frac{\alpha\sigma_{n-1}}{h}\bigg)\\
&\nonumber=\frac{1}{h}\partial_t(\alpha\sigma_{n-1})
-\frac{1}{h^2}\alpha\sigma_{n-1}\partial_th
-\alpha c_{ij}\nabla_{ij}
\bigg(\frac{\alpha\sigma_{n-1}}{h}\bigg)\\
&\nonumber=\frac{1}{h}\bigg(\partial_t
(\alpha\sigma_{n-1})-\alpha c_{ij}\nabla_{ij}(\alpha\sigma_{n-1})\bigg)
-\bigg(\frac{\alpha\sigma_{n-1}}{h}\bigg)^2
+\frac{\gamma}{h}\alpha\sigma_{n-1}\\
&\ \ +\frac{1}{h^2}\alpha c_{ij}(\alpha\sigma_{n-1})\nabla_{ij}h
+\frac{2}{h^2}\alpha c_{ij}\nabla_i(\alpha\sigma_{n-1})\nabla_jh
-\frac{2}{h^3}\alpha c_{ij}(\alpha\sigma_{n-1})\nabla_ih\nabla_jh\\
&\nonumber=\bigg(\frac{n-1}{h}+\frac{\phi^\prime(h)}
{\phi(h)}\bigg)\frac{(\alpha\sigma_{n-1})^2}{h}
-\bigg(n-1+h\frac{\phi^\prime(h)}{\phi(h)}\bigg)
\frac{\gamma\alpha\sigma_{n-1}}{h}+2\alpha c_{ij}\frac{\nabla_jh}{h}\nabla_i\bigg
(\frac{\alpha\sigma_{n-1}}{h}\bigg)\\
&\nonumber\ \ -pf\phi\sigma_{n-1}|\nabla\Psi|^{p-1}[(\nabla^2\Psi)\xi\cdot
(\alpha\sigma_{n-1}-\gamma h)_ie_i+(\alpha\sigma_{n-1}-\gamma h)(\nabla^2\Psi)\xi\cdot\xi]\\
&\nonumber\ \ -pf\phi\sigma_{n-1}|\nabla\Psi|^{p}
[(|\nabla\Psi|^{-2}\nabla\Psi\nabla^2\Psi\cdot\xi)
(\alpha\sigma_{n-1}-\gamma h)+(\alpha\sigma_{n-1}-\gamma h)].
\end{align}

Moreover, since $r^2=h^2+|\nabla h|^2$, then
\begin{align*}
\partial_t\bigg(\frac{r^2}{2}\bigg)&
=\partial_t\bigg(\frac{h^2}{2}\bigg)
+\partial_t\bigg(\frac{|\nabla h|^2}{2}\bigg)\\
&=h\partial_th+\sum\nabla_ih\nabla_i(\partial_th)\\
&=h\alpha\sigma_{n-1}
+\sigma_{n-1}\sum\nabla_ih\nabla_i\alpha
+\alpha\sum\nabla_ih\nabla_i\sigma_{n-1}-\gamma r^2,
\end{align*}
and
\begin{align*}
c_{ij}\nabla_{ij}\bigg(\frac{r^2}{2}\bigg)
&=c_{ij}(h\nabla_{ij}h+\nabla_ih\nabla_jh
+\sum h_k\nabla_ih_{kj}+\sum h_{ik}h_{jk})\\
&=h[(n-1)\sigma_{n-1}-c_{ij}\delta_{ij}h]
+c_{ij}\nabla_ih\nabla_jh+\sum h_k\nabla_k\sigma_{n-1}\\
&\ \ -c_{ij}\nabla_ih\nabla_jh+\sum c_{ij}b_{ik}b_{jk}+c_{ij}\delta_{ij}h^2-2(n-1)h\sigma_{n-1}\\
&=-(n-1)h\sigma_{n-1}+\sum h_k\nabla_k\sigma_{n-1}+\sum c_{ij}b_{ik}b_{jk}.
\end{align*}
Thus
\begin{align}\label{4.11}
\partial_t\bigg(\frac{r^2}{2}\bigg)-\alpha c_{ij}\nabla_{ij}\bigg(\frac{r^2}{2}\bigg)
=nh\alpha\sigma_{n-1}+\sigma_{n-1}\sum\nabla_ih\nabla_i\alpha
-\alpha\sum c_{ij}b_{ik}b_{jk}-\gamma r^2.
\end{align}
Using the arithmetic-geometric mean inequality, we have
$$\sum c_{ij}b_{ik}b_{jk}\geq l_0\sigma_{n-1}^{1+\frac{1}{n-1}}$$
for some positive constant $l_0$.

Since $c_{ij}\frac{\nabla_jh}{h}\nabla_i
(\frac{\alpha\sigma_{n-1}}{h})\leq0$ from (\ref{4.6}), and
$$\nabla_i(\alpha\sigma_{n-1}-\gamma h)=\frac{1}{h}\alpha\sigma_{n-1}\nabla_ih
-\frac{\lambda\alpha\sigma_{n-1}}
{1-\lambda\frac{r^2}{2}}r\nabla_ir
-\gamma\nabla_ih,$$
it follows from (\ref{4.8.1}), (\ref{4.10}) and (\ref{4.11}) that
\begin{align*}
\partial_t\Theta&\leq\frac{1}{1-\lambda\frac{r^2}{2}}
\bigg\{\bigg(\frac{n-1}{h}+\frac{\phi^\prime}{\phi}\bigg)
\frac{(\alpha\sigma_{n-1})^2}{h}
-(n-1+h\frac{\phi^\prime}{\phi})
\frac{\gamma\alpha\sigma_{n-1}}{h}\\
&\ \ \ -\frac{p\alpha\sigma_{n-1}}{h|\nabla\Psi|}
\bigg[\bigg(\frac{\nabla_ih}{h}\alpha\sigma_{n-1}
-\frac{\lambda r\alpha\sigma_{n-1}\nabla_ir}
{1-\lambda\frac{r^2}{2}}-\gamma\nabla_ih\bigg)\\
&\ \ \ \ \ \ \ \ \ (\nabla^2\Psi)\xi\cdot e_i
+(\alpha\sigma_{n-1}-\gamma h)
(\nabla^2\Psi)\xi\cdot\xi\bigg]\\
&\ \ \ -\frac{p\alpha\sigma_{n-1}}{h}
\bigg((|\nabla\Psi|^{-2}\nabla\Psi\nabla^2\Psi\cdot\xi)
(\alpha\sigma_{n-1}-\gamma h)
+(\alpha\sigma_{n-1}-\gamma h)\bigg)\bigg\}\\
&\ \ \ +\frac{\lambda}{(1-\lambda\frac{r^2}{2})^2}
\frac{\alpha\sigma_{n-1}}{h}\bigg[nh\alpha\sigma_{n-1}
-\alpha\sum\nabla_kh\nabla_k\sigma_{n-1}
-\alpha\sum c_{ij}b_{ik}b_{ik}-\gamma r^2\bigg],
\end{align*}
i.e.,
\begin{align}\label{4.12}
\nonumber\partial_t\Theta&\leq\bigg[\frac{(\nabla_ih+2ph)p\lambda
|\nabla^2\Psi|}{|\nabla\Psi|}+p\gamma h\bigg]
\frac{\alpha\sigma_{n-1}}{h(1-\lambda\frac{r^2}{2})}\\
&\ \ \ +\bigg[n+h\frac{\phi^\prime}{\phi}
+\frac{\lambda h(pr\nabla_ir|\nabla^2\Psi|
+nh|\nabla\Psi|)}{|\nabla\Psi|}\bigg]
\bigg(\frac{\alpha\sigma_{n-1}}
{h(1-\lambda\frac{r^2}{2})}\bigg)^2\\
\nonumber&\ \ \
-l_0\lambda\bigg(\frac{h^n}{\alpha}\bigg)^{\frac{1}{n-1}}
\bigg(\frac{\alpha\sigma_{n-1}}
{h(1-\lambda\frac{r^2}{2})}\bigg)^{2+\frac{1}{n-1}}.
\end{align}

Since $\varphi^\prime(h)=\frac{1}{\phi(h)}>0$ implies that $\varphi(h)$ strictly increasing. Combining with Lemma \ref{l4.1}, we know that $\varphi(h)$ has positive upper and lower bounds, which also shows that $\phi(h)$ has positive upper and lower bounds.
Using the previous estimates in Section 4.1, the definition of $\lambda$, and $p\in(1,n)$, we find that there exist positive constants $l_1$ depending on Lemmas \ref{l4.1}, \ref{l4.3} and \ref{l4.4}, as well as Corollary \ref{c4.2} such that
$$\frac{(\nabla_ih+2ph)p\lambda
|\nabla^2\Psi|}{|\nabla\Psi|}+p\gamma h\leq l_1,$$
and that there exists a positive constant $l_2$ depending on Lemmas \ref{l4.1}, \ref{l4.3} and \ref{l4.4} such that
$$n+h\frac{\phi^\prime}{\phi}
+\frac{\lambda h(pr\nabla_ir|\nabla^2\Psi|
+nh|\nabla\Psi|)}{|\nabla\Psi|}\leq l_2,$$
and that there exists a positive constant $l_3$ depending on Lemmas \ref{l4.1} and \ref{l4.4} such that
$$l_0\lambda\bigg(\frac{h^n}{\alpha}\bigg)^{\frac{1}{n-1}}
\geq l_3.$$

Therefore, (\ref{4.12}) can be further estimated as
$$\partial_t\Theta\leq l_1\Theta+l_2\Theta^2-l_3\Theta^{2+\frac{1}{n-1}}.$$
By the maximum principle, then we have
$$\Theta(\hat{\xi}_{\hat{t}},\hat{t})\leq L_2,$$
for some positive constant $L_2$ independent of $t$.
Since $\sigma_{n-1}=G^{-1}$, thus we obtain a uniformly positive lower bound for Gauss curvature.
\end{proof}

From  Lemma \ref{l4.1},
as discussed in Section 2 (or see \cite{U}), we know that the eigenvalues of  matrix $\{b_{ij}\}$ are positive, i.e. $\{b_{ij}\}$ is positive definite, and the principal curvatures are the eigenvalues of $\{b^{ij}\}$.
Therefore, to derive a positive upper bound of principal curvatures of $F(\cdot,t)$, it is equivalent to estimate the upper bound of the eigenvalues of $\{b^{ij}\}$.

\begin{lemma}\label{l4.6}
Under the assumptions of Lemma \ref{l4.1}. Then the principal curvature
$$\kappa_i(\cdot,t)\leq L_3,\ \ i=1,\cdots,n-1,$$
where $L_3$ is a positive constant independent of $t$.
\end{lemma}

\begin{proof}
We study the following auxiliary function
\begin{align}\label{4.13}
\overline{\mathcal{E}}(\xi,t)=\log\zeta_{\max}
(\{b^{ij}\})-a\log h+\frac{s}{2}r^2,
\end{align}
where $a$ and $s$ are positive constants to be specified later, and $\zeta_{\max}$ is the maximal eigenvalue of $\{b^{ij}\}$.
We suppose the spatial maximum of $\overline{\mathcal{E}}(\xi,t)$ is attained at $\xi_0\in\mathcal{S}^{n-1}$ for $t>0$. By a rotation of coordinates, suppose $\{b^{ij}(\xi_0,t)\}$ is diagonal, and $\zeta_{\max}=b^{11}(\xi_0,t)$. Then, (\ref{4.13}) can be rewritten as
$$\mathcal{E}(\xi,t)=\log b^{11}-a\log h+\frac{s}{2}r^2.$$
Next it is sufficient to prove that $\mathcal{E}(\cdot,t)$ has a positive upper bound.
For convenience, we write $\nabla_{ij}h=h_{ij}$ and $\nabla_{ij}r=r_{ij}$.

At the point $\xi_0$, we have
\begin{align}\label{4.14}
\nonumber0=\nabla_i\mathcal{E}
&=-b^{11}\nabla_ib_{11}-a\frac{h_i}{h}+srr_i\\
&=-b^{11}\nabla_i(h_{11}+h\delta_{11})-a\frac{h_i}{h}+srr_i.
\end{align}
and
\begin{align}\label{4.15}
0\geq\nabla_{ij}\mathcal{E}=-b^{11}
\nabla_{ij}b_{11}+(b^{11})^2(\nabla_ib_{11})^2
-a\bigg(\frac{h_{ij}}{h}-\frac{h_i^2}
{h^2}\bigg)+sr_i^2+srr_{ij}.
\end{align}
Furthermore, for $t>0$,
\begin{align}\label{4.16}
\nonumber\partial_t\mathcal{E}&=-b^{11}\partial_tb_{11}
-a\frac{\partial_th}{h}+sr\partial_tr\\
&=-b^{11}((\partial_th)_{11}+\partial_th)
-a\frac{\partial_th}{h}+sr\partial_tr.
\end{align}

From Equation (\ref{3.3}), we write
\begin{align}\label{4.17}
\log(\partial_th+\gamma h)=\log\sigma_{n-1}+\Lambda(\xi,t),
\end{align}
where
$$\Lambda(\xi,t)=\log(fh\phi(h)|\nabla\Psi|^p).$$
Differentiating (\ref{4.17}),
\begin{align}\label{4.18}
\frac{(\partial_th)_j+\gamma h_j}{\partial_th+\gamma h}
=\sum b^{ik}\nabla_jb_{ik}+\nabla_j\Lambda,
\end{align}
and
\begin{align}\label{4.19}
\frac{(\partial_th)_{11}+\gamma h_{11}}{\partial_th+\gamma h}
-\frac{(\gamma h_1+(\partial_th)_1)^2}{(\partial_th+\gamma h)^2}
=\sum b^{ii}\nabla_{11}b_{ii}-\sum b^{ii}b^{jj}(\nabla_1b_{ij})^2+\nabla_{11}\Lambda.
\end{align}
By the Ricci identity on $\mathcal{S}^{n-1}$
$$\nabla_{11}b_{ij}=\nabla_{ij}b_{11}-\delta_{ij}b_{11}
+\delta_{11}b_{ij}-\delta_{1i}b_{1j}+\delta_{1j}b_{1i},$$
and (\ref{4.15}), (\ref{4.16}), (\ref{4.18}) and (\ref{4.19}), at $\xi_0$, then
\begin{align}\label{4.20}
\nonumber\frac{\partial_t\mathcal{E}}{\partial_th+\gamma h}
&=\frac{-b^{11}((\partial_th)_{11}+\partial_th)}
{\partial_th+\gamma h}
-a\frac{\partial_th}{h(\partial_th+\gamma h)}
+s\frac{r\partial_tr}{\partial_th+\gamma h}\\
&\nonumber=-b^{11}\bigg(\frac{(\partial_th)_{11}+\gamma h_{11}
-\gamma h_{11}-\gamma h+\gamma h+\partial_th}{\partial_th+\gamma h}\bigg)\\
&\nonumber\ \ \ -a\frac{\partial_th}{h(\partial_th+\gamma h)}
+s\frac{r\partial_tr}{\partial_th+\gamma h}\\
&\nonumber=-b^{11}\frac{(\partial_th)_{11}+\gamma h_{11}}{\partial_th+\gamma h}
+\frac{\gamma}{\partial_th+\gamma h}-b^{11}-\frac{a}{h}+\frac{a\gamma}{\partial_th+\gamma h}+
s\frac{r\partial_tr}{\partial_th+\gamma h}\\
&\nonumber\leq-b^{11}\sum b^{ii}\nabla_{11}b_{ii}+b^{11}\sum b^{ii}b^{jj}(\nabla_1b_{ij})^2
-b^{11}\nabla_{11}\Lambda\\
&\ \ \ +\frac{\gamma(1+a)}{\partial_th+\gamma h}
+s\frac{r\partial_tr}{\partial_th+\gamma h}\\
&\nonumber=-b^{11}\sum b^{ii}(\nabla_{ii}b_{11}-b_{11}+b_{ii})
+b^{11}\sum b^{ii}b^{jj}(\nabla_1b_{ij})^2
-b^{11}\nabla_{11}\Lambda\\
&\nonumber\ \ \ +\frac{\gamma(1+a)}{\partial_th+\gamma h}
+s\frac{r\partial_tr}{\partial_th+\gamma h}\\
&\nonumber\leq-\sum b^{ii}(b^{11})^2(\nabla_ib_{11})^2+\sum b^{ii}a\bigg(\frac{h_{ii}}{h}
-\frac{h_i^2}{h^2}\bigg)-\sum b^{ii}sr_i^2\\
&\nonumber\ \ \ -\sum b^{ii}srr_{ii}+b^{11}\sum b^{ii}b^{jj}(\nabla_1b_{ij})^2
-b^{11}\nabla_{11}\Lambda\\
&\nonumber\ \ \ +\frac{\gamma(1+a)}{\partial_th+\gamma h}
+s\frac{r\partial_tr}{\partial_th+\gamma h}+\sum b^{ii}-(n-1)b^{11}\\
&\nonumber\leq-a\sum b^{ii}+\frac{(n-1)a}{h}-b^{11}\nabla_{11}\Lambda
+\frac{\gamma(1+a)}{\partial_th+\gamma h}\\
&\nonumber\ \ \ +s\bigg(\frac{r\partial_tr}{\partial_th+\gamma h}
-\sum b^{ii}(r_i^2+rr_{ii})\bigg),
\end{align}
where
\begin{align}\label{4.21}
\nonumber&\partial_tr=\frac{h\partial_th+\sum h_k(\partial_th)_k}{r},\\
&r_i=\frac{hh_i+\sum h_kh_{ki}}{r}=\frac{h_ib_{ii}}{r},\\
\nonumber&r_{ij}=\frac{hh_{ij}+h_ih_j+\sum h_kh_{kij}+\sum h_{ki}h_{kj}}{r}-\frac{h_ih_jb_{ii}b_{jj}}{r^3}.
\end{align}
Thus
\begin{align}\label{4.22}
&\nonumber\frac{r\partial_tr}{\partial_th+\gamma h}
-\sum b^{ii}(r_i^2+rr_{ii})\\
&=\frac{h\partial_th}{\partial_th+\gamma h}-h\sum b^{ii}h_{ii}-b^{ii}\sum h_{ii}^2
-\frac{\gamma|\nabla h|^2}{\partial_th+\gamma h}+\sum h_k\nabla_k\Lambda\\
&\nonumber =nh-\frac{\gamma r^2}{\partial_th+\gamma h}
-\sum b_{ii}+\sum h_k\nabla_k\Lambda.
\end{align}
Substituting (\ref{4.22}) into (\ref{4.20})
\begin{align}\label{4.23}
\nonumber\frac{\partial_t\mathcal{E}}{\partial_th+\gamma h}
&\leq -a\sum b^{ii}+nh(a+s)+\frac{\gamma(1+a-sr^2)}{\partial_th+\gamma h}
-s\sum b_{ii}\\
&\ \ \ -b^{11}\nabla_{11}\Lambda+s\sum h_k\nabla_k\Lambda.
\end{align}

Next we calculate $-b^{11}\nabla_{11}\Lambda$ and $s\sum h_k\nabla_k\Lambda$.
From the expression for $\Lambda(\xi,t)$, we have
$$\nabla_k\Lambda=\frac{f_k}{f}
+\frac{h_k}{h}+\frac{\phi^\prime(h)}{\phi(h)}h_k
+p\frac{|\nabla\Psi|_k}{|\nabla\Psi|},$$
and
\begin{align*}
\nabla_{kl}\Lambda&=\frac{ff_{kl}-f_kf_l}{f^2}
+\frac{hh_{kl}-h_kh_l}{h^2}+\frac{\phi^{\prime\prime}
h_kh_l+\phi^\prime h_{kl}}{\phi}-\frac{(\phi^\prime)^2h_kh_l}{\phi^2}\\
&\ \ \ +p\frac{|\nabla\Psi|_{kl}}{|\nabla\Psi|}
-p\frac{|\nabla\Psi|_{k}|\nabla\Psi|_{l}}{|\nabla\Psi|^2}.
\end{align*}
Recall that
$$|\nabla\Psi(F,t)|=-\nabla\Psi(F,t)\cdot\xi.$$
Taking the covariant derivative above
$$|\nabla\Psi|_j=-b_{ij}((\nabla^2\Psi)e_i\cdot\xi).$$
So
\begin{align}\label{4.24}
|\nabla\Psi|_{11}&=-b_{i11}((\nabla^2\Psi)e_i\cdot\xi)
-b_{j1}b_{i1}
((\nabla^3\Psi)e_je_i\cdot\xi)\\
&\nonumber\ \ +b_{i1}((\nabla^2\Psi)\xi\cdot\xi)
-b_{i1}((\nabla^2\Psi)e_i\cdot e_1).
\end{align}
It follows that
\begin{align}\label{4.25}
\nonumber s\sum h_k\nabla_k\Lambda&=
s\sum h_k\bigg(\frac{f_k}{f}+\frac{h_k}{h}
+\frac{\phi^\prime}{\phi}h_k\bigg)
-ps\frac{h_k}{|\nabla\Psi|}
((\nabla^2\Psi)e_k\cdot\xi)b_{kk}\\
&\leq c_1s-ps\frac{h_k}{|\nabla\Psi|}
((\nabla^2\Psi)e_k\cdot\xi)b_{kk},
\end{align}
and
\begin{align}\label{4.26}
\nonumber-b^{11}\nabla_{11}\Lambda
&=-b^{11}\bigg[\frac{ff_{11}-f_1^2}{f^2}
+\frac{hh_{11}-h_1^2}{h^2}
+\frac{\phi^{\prime\prime}h_1^2+\phi^{\prime}h_{11}}{\phi}
-\frac{(\phi^{\prime})^2h_1^2}{\phi^2}\bigg]\\
\nonumber&\ \ \ -pb^{11}\frac{|\nabla\Psi|_{11}}{|\nabla\Psi|}
+pb^{11}\frac{(|\nabla\Psi|_{1})^2}{|\nabla\Psi|^2}\\
&\leq c_2b^{11}+c_3+c_4b_{11}+pb^{11}b_{i11}
\frac{(\nabla^2\Psi)e_i\cdot\xi}{|\nabla\Psi|}.
\end{align}
where $c_1, c_2, c_3$ and $c_4$ are positive constants independent of $t$.

From (\ref{4.14}) and (\ref{4.21}), we have
$$b^{11}b_{i11}=-a\frac{h_i}{h}+srr_i
=-a\frac{h_i}{h}+sh_ib_{ii}.$$
This, together with (\ref{4.26}), yields
$$-b^{11}\nabla_{11}\Lambda\leq c_2b^{11}+c_3+c_4b_{11}+c_5s\sum b_{ii}+c_6a,$$
where $c_5$ and $c_6$ are positive constants independent of $t$. Hence
\begin{align}\label{4.27}
-b^{11}\nabla_{11}\Lambda+s\sum h_k\nabla_k\Lambda
\leq \widehat{c}_1s+\widehat{c}_2a+\widehat{c}_3b^{11}
+\widehat{c}_4b_{11}
+\widehat{c}_5s\sum b_{ii}+\widehat{c}_6.
\end{align}
 Substituting (\ref{4.27}) into (\ref{4.23}), if we choose $s\gg a$, then
\begin{align*}
\frac{\partial_t\mathcal{E}}{\partial_th+\gamma h}
&\leq-a\sum b^{ii}+nh(a+s)-s\sum b_{ii}+\widehat{c}_3b^{11}
 +\widehat{c}_4b_{11}+\widehat{c}_5s\sum b_{ii}+\widehat{c}_6.
\end{align*}
Furthermore, let $a>\widehat{c}_3$, and $b^{ii}$ be large enough, then
$$\frac{\partial_t\mathcal{E}}{\partial_th+\gamma h}<0.$$
Therefore
$$\mathcal{E}(\xi_0,t)=\overline{\mathcal{E}}(\xi_0,t)\leq L_3,$$
for some positive constant $L_3$ independent of $t$. The proof is completed.
\end{proof}

As a consequence of Lemma \ref{l4.5} and Lemma \ref{l4.6}, we can obtain the following corollary.

\begin{corollary}\label{c4.7}
Under the assumptions of Lemma \ref{l4.1}. Then the principal curvature
$$L_4\leq\kappa_i(\cdot,t)\leq L_3,\ \ i=1,\cdots,n-1,$$
for $\forall(\cdot,t)\in\mathcal{S}^{n-1}\times(0,T)$. Here $L_4$ is a positive constant independent of $t$.
\end{corollary}

\section{Proofs of main theorems}

\subsection*{Proof of Theorem \ref{t1.2}}
From the $C^2$-estimates obtained in Corollary \ref{c4.7}, we know that Equation (\ref{3.3}) is uniformly parabolic on any finite time interval and has the short time existence.
By $C^0, C^1$ and $C^2$-estimates (Lemmas \ref{l4.1}, \ref{l4.3} and Corollary \ref{c4.7}), and the Krylov's theory \cite{K}, we get the H\"{o}lder continuity of $\nabla^2h$ and $\partial_th$.
Then we get the higher order derivation estimates by the regularity theory of the uniformly parabolic equations. Therefore, we obtain the long-time existence and regularity of the solution to Equation (\ref{3.3}).
Moreover, we have
\begin{align}\label{5.01}
\|h\|_{C_{\xi,t}^{i,j}(\mathcal{S}^{2n}\times[0,T))}\leq C
\end{align}
for some $C>0$, independent of $t$, and for each pairs of nonnegative integers $i$ and $j$.

With the aid of the Arzel\`{a}-Ascoli theorem and a diagonal argument, there exists a sequence of $t$, denoted by $\{t_k\}_{k\in\mathbb{N}}\subset(0,\infty)$, and a smooth function $h(\xi)$ such that
\begin{align}\label{5.02}
\|h(\xi,t_k)-h(\xi)\|_{C^i(\mathcal{S}^{2n})}\rightarrow0
\end{align}
uniformly for any nonnegative integer $i$ as  $t_k\rightarrow\infty$.
This illustrates that $h(\xi)$ is a support function. Let $\Omega$ be a convex body determined by $h(\xi)$, we conclude that $\Omega$ is smooth and strictly convex with the origin in its interior.

In the following, we prove that Equation (\ref{1.4}) has a non-symmetric smooth solution.
From Lemma \ref{l3.1}, we see that
\begin{align}\label{5.2}
\partial_tC_p(\Omega_t)\geq0.
\end{align}
If there exists a time $\tilde{t}$ such that
$$\partial_tC_p(\Omega_t)\bigg|_{t=\tilde{t}}=0,$$
then by the equality condition in Lemma \ref{l3.1}, we have
$$f\phi(h)|\nabla\Psi(F,\tilde{t})|^p\sigma_{n-1}
=\tau,$$
for some constant $\tau>0$, that is, support function $h(\xi,\tilde{t})$ of $\Omega_{\tilde{t}}$ satisfies Equation (\ref{1.4}).

Next we verify the case of $\partial_tC_p(\Omega_t)>0$.
From the proof in Corollary \ref{c4.2}, there exists a positive constant $\mathcal{L}$ which is independent of $t$ such that
\begin{align}\label{5.2.1}
C_p(\Omega_t)\leq\mathcal{L},
\end{align}
and  $\partial_tC_p(\Omega_t)$ is uniformly continuous.
Combining (\ref{5.2}) and (\ref{5.2.1}), and applying the Fundamental Theorem of calculus, we obtain
$$\int_0^tC_p^\prime(\Omega_t)dt
=C_p(\Omega_t)-C_p(\Omega_0)
\leq C_p(\Omega_t)\leq\mathcal{L},$$
which leads to
$$\int_0^\infty C_p^\prime(\Omega_t)dt<\mathcal{L}.$$
This implies that there exists a subsequence of time $t_j\rightarrow\infty$ such that
$$\lim_{t_j\rightarrow\infty}\partial_tC_p(\Omega_{t_j})=0.$$

From the proof of Lemma \ref{l3.1}, we have
\begin{align*}
\partial_tC_p(\Omega_t)\bigg|_{t=t_j}
=(p-1)\int_{\mathcal{S}^{n-1}}
|\nabla\Psi(F,t)|^p\sigma_{n-1}\partial_thd\xi\bigg|_{t=t_j}.
\end{align*}
Passing to the limit, we have
\begin{align*}
0&=\lim_{t_j\rightarrow\infty}\partial_t
C_p(\Omega_t)\bigg|_{t=t_j}\\
&=\frac{p-1}{\int_{\mathcal{S}^{n-1}}
\frac{h_\infty}{f\phi(h_\infty)}d\xi}
\bigg[\int_{\mathcal{S}^{n-1}}fh_\infty\phi(h_\infty)
|\nabla\Psi|^{2p}\widetilde{\sigma}_{n-1}^2d\xi
\int_{\mathcal{S}^{n-1}}\frac{h_\infty}{f\phi(h_\infty)}d\xi\\
&\ \ \ -\bigg(\int_{\mathcal{S}^{n-1}}h_\infty
|\nabla\Psi|^{p}\widetilde{\sigma}_{n-1}d\xi\bigg)^2\bigg]\\
&\geq\frac{p-1}{\int_{\mathcal{S}^{n-1}}
\frac{h_\infty}{f\phi(h_\infty)}d\xi}
\bigg[\bigg(\int_{\mathcal{S}^{n-1}}h_\infty
|\nabla\Psi|^{p}\widetilde{\sigma}_{n-1}d\xi\bigg)^2
-\bigg(\int_{\mathcal{S}^{n-1}}h_\infty
|\nabla\Psi|^{p}\widetilde{\sigma}_{n-1}d\xi\bigg)^2\bigg]\\
&=0,
\end{align*}
this means that
$$f\phi(h_\infty)|\nabla\Psi|^p
\widetilde{\sigma}_{n-1}=\tau,$$
for some constant $\tau>0$, where $h_\infty$, $\widetilde{\sigma}_{n-1}$ are the support function and product of the principal curvature radii of the limit convex hypersurface $\Omega_\infty$, respectively.
The proof of Theorem \ref{t1.2} is completed.   \hfill$\square$

\subsection*{Proof of Theorem \ref{t1.4}}
Let $h_1$ and $h_2$ be two solutions of equation (\ref{1.4}).
We first prove the following fact:
\begin{align}\label{5.8}
\max\frac{h_1}{h_2}\leq1.
\end{align}
We use proof by contradiction.
Suppose (\ref{5.8}) is not true, namely $\max\frac{h_1}{h_2}>1$.
Suppose $\max\frac{h_1}{h_2}$ attains at $z_0\in\mathcal{S}^{n-1}$, then $h_1(z_0)>h_2(z_0)$.
Let $\mathcal{P}=\log\frac{h_1}{h_2}$, we have at $z_0$
$$0=\nabla\mathcal{P}=\frac{\nabla h_1}{h_1}
-\frac{\nabla h_2}{h_2},$$
and
$$0\geq\nabla^2\mathcal{P}=\frac{\nabla^2 h_1}{h_1}
-\frac{\nabla^2 h_2}{h_2}.$$

By the equation (\ref{1.4}), and the homogeneity of $p$-capacitary measure (see \cite{CNS}), one has
\begin{align*}
1&=\frac{\phi(h_2)|\nabla\Psi(\nabla h_2)|^p\det(\nabla^2h_2+h_2I)}
{\phi(h_1)|\nabla\Psi(\nabla h_1)|^p\det(\nabla^2h_1+h_1I)}\\
&=\frac{\phi(h_2)h_2^{n-p-1}\det(\frac{\nabla^2h_2}{h_2}+I)}
{\phi(h_1)h_1^{n-p-1}\det(\frac{\nabla^2h_1}{h_1}+I)}\\
&\geq\frac{\phi(h_2)h_2^{n-p-1}\det(\frac{\nabla^2h_1}{h_1}+I)}
{\phi(h_1)h_1^{n-p-1}\det(\frac{\nabla^2h_1}{h_1}+I)}\\
&=\frac{\phi(h_2)h_2^{n-p-1}}{\phi(h_1)h_1^{n-p-1}}.
\end{align*}
Let $h_2(z_0)=\delta h_1(z_0)$, then we have
$$\phi(\delta h_1)\leq\delta^{p+1-n}\phi(h_1).$$
Since $\delta\geq1$, it follows that $h_2(z_0)\geq h_1(z_0)$.  This is a contradiction. Thus (\ref{5.8}) holds.

Interchanging $h_1$ and $h_2$, (\ref{5.8}) implies
\begin{align*}
\max\frac{h_2}{h_1}\leq1.
\end{align*}
Combining it with (\ref{5.8}), we have $h_1\equiv h_2$. The proof of Theorem \ref{t1.4} is completed.
\hfill$\square$

\subsection*{Proof of Theorem \ref{t1.3}}
Let $\varphi$ be as in Theorem \ref{t1.3}, $\mu$ be a finite Borel measure on $\mathcal{S}^{n-1}$. Given a function $f: \mathcal{S}^{n-1}\rightarrow(0,\infty)$, we define the measure as
$$d\mu_{f}=\frac{1}{f}d\xi.$$
Suppose $\phi$ is smooth. By the proof of Lemma 3.7 in \cite{CL}, there exists a family of positive function $\{f_k\}\subset C^{\infty}(\mathcal{S}^{n-1})$ so that $\mu_{f_k}\rightarrow\mu$ as $k\rightarrow\infty$, weakly.

Let $\Omega_{0,k}=B$ be the unit ball in $\mathbb{R}^n$.
For a smooth, closed, and strictly convex hypersurfaces $\Omega_{t,k}$, its support function satisfies the flow (\ref{3.3}) and $h(\cdot,0)=1$.
From Theorem \ref{t1.2}, when $t\rightarrow\infty$, it can be known that the hypersurfaces $\Omega_{t,k}$  converges in $C^{\infty}$ to a smooth, closed, and strictly convex hypersurfaces $\Omega_{\infty,k}$, and satisfies
\begin{align}\label{5.3}
\phi(h_{\Omega_{\infty,k}})
d\mu_p(\Omega_{\infty,k},\cdot)
=\frac{n-p}{p-1}\frac{C_p(\Omega_{\infty,k})}
{\int_{\mathcal{S}^{n-1}}[h_{\Omega_{\infty,k}}/
\phi(h_{\Omega_{\infty,k}})]d\mu_{f_k}}d\mu_{f_k}.
\end{align}

In the following, we need to obtain the uniform upper and lower bounds for $h_{\Omega_{\infty,k}}$. Choose $v\in\mathcal{S}^{n-1}$, let $h_{\overline{v}}$ be the support function of the line segment joining $\pm v$.
It follows from Lemma 3.6 and Corollary 3.7 in \cite{HH} that there exists a constant $d>0$ such that
\begin{align}\label{5.4}
\min_{v\in\mathcal{S}^{n-1}}\|h_{\overline{v}}
\|_{\varphi,\mu_{f_k}}\geq d,
\end{align}
for all $k$. For any $v$ and $R_k$, there is $\pm R_kv\in\partial\Omega_{\infty,k}$ with maximal distance from the origin, thus we have $R_kh_{\overline{v}}(\xi)\leq h(\Omega_{\infty,k},\xi)$ for all $\xi\in\mathcal{S}^{n-1}$.

Furthermore, we define
$$\Phi_k(t)=\frac{1}{|\mu_{f_k}|}\int_{\mathcal{S}^{n-1}}
\varphi(h_{\Omega_{t,k}})d\mu_{f_k}.$$
By Lemma \ref{l3.2}, we have $\frac{d}{dt}\Phi_k(t)=0$, it follows that $\Phi_k(t)=\Phi_k(0)=\varphi(1)$. From  (\ref{2.4}), it suffices to have
$$\|h_{\Omega_{t,k}}\|_{\varphi,\mu_{f_k}}\leq1.$$
Combining (\ref{2.5}) with Lemma 4 in \cite{HLY}, one has
\begin{align}\label{5.5}
R_k\min_{v\in\mathcal{S}^{n-1}}\|h_{\overline{v}}
\|_{\varphi,\mu_{f_k}}
\leq R_k\|h_{\overline{v}}\|_{\varphi,\mu_{f_k}}
\leq\|h_{\Omega_{\infty,k}}\|_{\varphi,\mu_{f_k}}\leq1.
\end{align}
Thus the uniform upper bound of $h_{\Omega_{\infty,k}}$ can be obtained from (\ref{5.4}) and (\ref{5.5}).

By Lemma 1 in \cite{LXZ}, Lemma \ref{l3.1}, and the upper bound of $h_{\Omega_{\infty,k}}$, we get for $p\in(1,n)$,
$$S_p(\Omega_{\infty,k})
\geq\bigg(\frac{p-1}{n-p}\bigg)^{p-1}C_p(\Omega_{\infty,k})
\geq\bigg(\frac{p-1}{n-p}\bigg)^{p-1}C_p(\Omega_{0,k})=c_0,$$ where $\Omega_{0,k}$ is the unit ball $B$, and $c_0$ is a positive constant depending on $p$ and $C_p(B)$.
This means that $h_{\Omega_{\infty,k}}$
has a uniform lower bound.
Hence, we can find some positive constants $c_1,c_2$ independent of $k$, such that
$$c_1\leq h_{\Omega_{\infty,k}}\leq c_2.$$
Therefore, there are positive numbers $c$ and $C$ depending on $c_1$ and $c_2$ such that
$$c\leq\int_{\mathcal{S}^{n-1}}\frac{h_{\Omega_{\infty,k}}}
{\phi(h_{\Omega_{\infty,k}})}d\mu_{f_k}\leq C,$$
for $k$ being large enough.

By the Blaschke selection theorem, we obtain that $\Omega_{\infty,k}$ subsequently converges to a convex hypersurface $\Omega_{\infty,0}$.
Taking the limit $k\rightarrow\infty$ in (\ref{5.3}), combining the positive homogeneity and weak convergence of $p$-capacitary measure (see \cite{CNS}), we can find a convex body $\Omega$  generated by $\Omega_{\infty,0}$ such that
$$\lambda\phi(h_{\Omega})d\mu_p(\Omega,\cdot)=d\mu$$
for some positive constant $\lambda$.
Thus $\Omega$ is the desired solution. A further approximation allows us to confirm that $\phi$ is merely continuous.   \hfill$\square$

\subsection*{Acknowledgments}

The authors would like to express their heartfelt thanks to Professors X Zhang and X Cao for their helpful comments and suggestions.

\end{document}